\documentclass{amsart}
\usepackage{amsmath, amscd, amssymb, amsthm}
\usepackage{bbm}
\usepackage{latexsym}
\usepackage{amsfonts}
\usepackage{graphicx}
\usepackage[all,cmtip]{xy}
\usepackage[colorlinks,linkcolor=blue,breaklinks=blue,urlcolor=blue,citecolor=blue,anchorcolor=blue,pagebackref]{hyperref}%
\usepackage{geometry}
\newtheorem{theorem}{Theorem}
\newtheorem{lemma}{Lemma}

\newtheorem{proposition}{Proposition}

\newtheorem{conjecture}{Conjecture}

\renewcommand*\backref[1]{}
\renewcommand*\backrefalt[4]{ \ifcase #1 \or (cited on page #2) \else (cited on pages #2) \fi}

\newcommand{\be}{\begin{equation}}
\newcommand{\ee}{\end{equation}}
\newcommand{\bea}{\begin{eqnarray}}
\newcommand{\eea}{\end{eqnarray}}

\newcommand{\vs}{\vspace{0.5cm}}

\def\XXint#1#2#3{{\setbox0=\hbox{$#1{#2#3}{\int}$ }
\vcenter{\hbox{$#2#3$ }}\kern-.6\wd0}}

\begin{document}

\title[Streets-Tian Conjecture on several special types of Hermitian manifolds]{Streets-Tian Conjecture on several special types of Hermitian manifolds}

\author{Yuqin Guo}
\address{Yuqin Guo. School of Mathematical Sciences, Chongqing Normal University, Chongqing 401331, China}
\email{{1942747285@qq.com}}\thanks{The corresponding author Zheng is partially supported by NSFC grants 12141101 and 12471039, by Chongqing grant cstc2021ycjh-bgzxm0139, by Chongqing Normal University grant 24XLB026, and is supported by the 111 Project D21024.}

\author{Fangyang Zheng}
\address{Fangyang Zheng. School of Mathematical Sciences, Chongqing Normal University, Chongqing 401331, China}
\email{20190045@cqnu.edu.cn; franciszheng@yahoo.com} \thanks{}


\begin{abstract}
A Hermitian-symplectic metric is a Hermitian metric whose K\"ahler form is given by the $(1,1)$-part of a closed $2$-form. Streets-Tian Conjecture states that a compact complex manifold admitting a Hermitian-symplectic metric must be K\"ahlerian (i.e., admitting a K\"ahler metric). The conjecture is known to be true in complex dimension $2$ but is still open in complex dimensions $3$ or higher. In this article, we confirm the conjecture for some special types of compact Hermitian manifolds, including Chern K\"ahler-like manifolds, non-balanced Bismut torsion parallel (BTP) manifolds, and compact quotients of Lie groups whose Lie algebra  contains a $J$-invariant abelian ideal of codimension $2$. The last type is a natural generalization to (compact quotients of) almost abelian Lie groups. The non-balanced BTP case contains all Vaisman manifolds and all Bismut K\"ahler-like manifolds as  subsets. These results extend some of the earlier works on the topic by Fino, Kasuya, Vezzoni, Angella, Otiman, Paradiso, and others. Our approach is elementary in nature, by giving explicit descriptions of Hermitian-symplectic metrics on such spaces as well as the pathways of deforming them into K\"ahler ones, aimed at illustrating the algebraic complexity and subtlety of Streets-Tian Conjecture.
\end{abstract}

\maketitle

\tableofcontents

\section{Introduction and statement of results}\label{intro}

In \cite{ST}, Streets and Tian introduced the notion of {\em Hermitian-symplectic metric,} which is a Hermitian metric $g$ on a compact complex manifold  $M^n$ such that its K\"ahler form $\omega$ is the $(1,1)$-part of a closed $2$-form. In other words, there exists a global $(2,0)$-form $\alpha$ on $M^n$ so that $\Omega = \alpha + \omega + \overline{\alpha}$ is closed. Equivalently, this means a compact complex manifold $M^n$ admitting a symplectic form (i.e., non-degenerate closed $2$-form) $\Omega$ so that $\Omega (x, Jx)>0$ for any non-zero tangent vector $x$. For this reason, a Hermitian-symplectic structure is sometimes also called a {\em symplectic structure taming a complex structure $J$} (\cite{EFV}). 

Hermitian-symplectic metrics are interesting in many aspects. First of all, it is a natural mix of Hermitian and symplectic structures, two of the classic objects of study in geometry. Secondly, from the definition, it is easy to see that any Hermitian-symplectic metric would satisfy $\partial \overline{\partial} \omega =0$, namely, it is {\em pluriclosed} (also known as SKT in some literature, meaning {\em strong K\"ahler with torsion}), a type of special Hermitian metrics that are extensively studied (see for example the excellent survey \cite{FinoTomassini} and the references therein). Thirdly, Hermitian-symplectic metrics play a big role in the Hermitian curvature flow theory of Streets and Tian \cite{ST11}. For instance, the recent work of Ye \cite{Ye} shows that the Hermitian-symplectic  property is preserved under the pluriclosed flow of Streets and Tian, which evolves a Hermitian metric in the direction of the $(1,1)$-part of its first Bismut Ricci curvature tensor. Hermitian-symplectic metrics also enjoy some nice properties, for instance, they are  stable under small deformations, proved by S. Yang in \cite{SongYang}, and pluriclosed or Hermitian-symplectic metrics play a role in the theory of generalized complex geometry initiated by Hitchin which has been an active research area in recent years. See \cite{AG, CG, EFG, FinoTomassini09, GS, Hitchin, Streets, ST12} and the references therein for more on the topic.

\begin{conjecture}[{\bf Streets-Tian \cite{ST}}]
If a compact complex manifold admits a Hermitian-symplectic metric, then it must admit a K\"ahler metric.
\end{conjecture}
In other words, compact non-K\"ahlerian manifolds do not admit any Hermitian-symplectic metric. There are plenty of examples of Hermitian-symplectic metrics that are not K\"ahler, but so far all such examples live on K\"ahlerian manifolds, so the point of the conjecture is the hope that non-K\"ahelrian manifolds would not support such metrics and  one might even be able to `deform' a given Hermitian-symplectic metric into a K\"ahler one. 

The case of complex dimension $2$ (real dimension $4$) is rather special, and Streets-Tian Conjecture is known to be true (\cite{LiZhang, ST}). See also \cite{Donaldson} for a related more general conjecture in real dimension $4$. In complex dimension $3$ or higher, the conjecture is only known in some special cases. In the beautiful work \cite{Verbitsky}, Verbitsky showed that any non-K\"ahlerian twistor space does not admit any pluriclosed metric. In \cite{Chiose}, Chiose proved that any Fujiki ${\mathcal C}$ class manifold (namely, compact complex manifolds bimeromorphic to compact K\"ahler manifolds) does not admit any pluriclosed metric. In \cite{FuLiYau}, Fu, Li, and Yau proved that an important construction of a special type of non-K\"ahler Calabi-Yau threefolds does not admit any pluriclosed metric. So in particular, we know that  Streets-Tian Conjecture holds for all twistor spaces, all Fujiki ${\mathcal C}$ class manifolds, and all the special non-K\"ahler Calabi-Yau threefolds of Fu-Li-Yau. 

In their recent work \cite{AOt}, Angella and Otiman showed that any (non-K\"ahler) Vaisman manifold does not admit Hermitian-symplectic metrics (as well as a number of other interesting special types of Hermitian metrics). Note that {\em Vaisman manifolds} are a special type of {\em locally conformally K\"ahler} manifolds. To be more precise, a Hermitian manifold $(M^n,g)$ is Vaisman, if its K\"ahler form $\omega$ satisfies $d\omega = \psi \wedge \omega$ for a closed $1$-form $\psi$, and $\psi$ is parallel under the Levi-Civita connection. The book \cite{OV} forms an encyclopedia for locally conformally K\"ahler and Vaisman manifolds.  

Recall that a Hermitian metric is said to be {\em Chern K\"ahler-like,} it the Chern curvature tensor obeys all K\"ahler symmetries, namely, $R_{i\bar{j}k\bar{\ell}}=R_{k\bar{j}i\bar{\ell}} $ for any indices.  The notion was introduced in \cite{YangZ} for Chern and Levi-Civita connections and was generalized to any metric connections on a Hermitian manifold by \cite{AOUV}. Chern flat metrics are special examples of Chern K\"ahler-like metrics. Recall that a Hermitian manifold $(M^n,g)$ is said to be {\em astheno-K\"ahler,} if its K\"ahler form $\omega$ satisfies $\partial \overline{\partial} (\omega^{n-2})=0$. The notion was introduced by Jost and Yau in \cite{JostYau}. The first result of this article is the following observation:

\begin{proposition} \label{prop0}
Let $(M^n,g)$ be a compact Hermitian manifold. If $g$ is Chern K\"ahler-like but not K\"ahler, then $M^n$ cannot admit any Hermitian-symplectic metric or any astheno-K\"ahler metric. 
\end{proposition}

In particular Streets-Tian conjecture holds for all compact Chern K\"ahler-like manifolds. Note that when $g$ is Chern flat (\cite{Boothby}), the result is due to Di Scala-Lauret-Vezzoni \cite[Proposition 3.3]{DLV}, which states that any compact non-K\"ahler Chern flat manifold does not admit any pluriclosed metric. The Chern flat case is also a consequence of the result of Fino-Kasuya-Vezzoni \cite[Corollary 3.2]{FKV} which states that
compact Chern flat manifolds cannot admit any pluriclosed metric unless the complex Lie group is nilpotent (and Streets-Tian Conjecture is already established in the nilpotent case). Here we give a simple and elementary argument which covers the slightly more general Chern K\"ahler-like case. 

\vspace{0.1cm}

Next let us recall that {\em Bismut torsion parallel} (or BTP for brevity) manifolds are Hermitian manifolds whose Bismut connection (\cite{Bismut}) has parallel torsion. We refer the readers to \cite{ZhaoZ22, ZhaoZ24} for more discussions on this interesting type of special Hermitian manifolds. By the result of Andrada and Villacampa \cite{AndV}, Vasiman manifolds are always BTP. But Vaisman manifolds only form a relatively small subset in the set of BTP manifolds. The latter contains other examples such as {\em Bismut K\"ahler-like} manifolds (including all Bismut flat manifolds) and complex nilmanifolds that are BTP, when the dimension is $3$ or higher. See \cite{YZZ, ZZ-Crelle, ZZ-JGP} for more discussions on such metrics. The second result of this article is the following proposition which confirms the Streets-Tian Conjecture for all non-balanced BTP manifolds. 

\begin{proposition} \label{prop1}
Let $(M^n,g)$ be a compact Hermitian manifold. If $g$ is non-balanced and is Bismut torsion parallel, then $M^n$ cannot admit any Hermitian-symplectic metric.
\end{proposition}

Recall that a Hermitian metric $g$ is said to be {\em balanced} if its K\"ahler form satisfies $d(\omega^{n-1})=0$.  A large and interesting type of compact complex manifolds are the {\em Lie-complex} manifolds, meaning compact quotients of Lie groups by discrete subgroups, where the complex structure (when lifted onto the universal covers) is left-invariant.  When a Lie group $G$ admits a compact quotient, it will admit a bi-invariant measure (\cite{Milnor}). So by the averaging trick of Fino and Grantcharov \cite{FG, EFV}, if a Lie-complex manifold $M^n=G/\Gamma $ admits a  Hermitian-symplectic (or pluriclosed) metric, then it will admit a left-invariant Hermitian-symplectic (or pluriclosed) metric. So in order to verify Streets-Tian Conjecture on $M$, it suffices to consider left-invariant metrics on $G$, thus it reduces to the study of Hermitian-symplectic metrics on the corresponding Lie algebras.

An important  supporting evidence to Streets-Tian Conjecture is the following theorem of Enrietti, Fino and Vezzoni \cite{EFV} (see also \cite{AN} for an important supplement by Arroyo-Nicolini), which states that any complex nilmanifold cannot admit a Hermitian-symplectic metric, unless it is  a complex torus. 

\begin{theorem}[{\bf Enrietti-Fino-Vezzoni \cite{EFV}}]
Let $M^n=G/\Gamma $ be a complex nilmanifold, namely, a compact quotient of a nilpotent Lie group by a discrete subgroup where the complex structure (when lifted onto $G$) is left-invariant. If $M^n$ admits a Hermitian-symplectic metric, then $G$ must be abelian (hence $M$ is a complex torus).
\end{theorem}

So Streets-Tian Conjecture holds for all complex nilmanifolds.  In \cite{FKV}, Fino, Kasuya, Vezzoni confirmed Streets-Tian Conjecture for several special types of Lie-complex manifolds. In particular, their Theorem 1.1 implies that the conjecture holds for all {\em Oeljeklaus-Toma} manifolds, an interesting class of non-K\"ahler manifolds which are high dimensional analogues of Inoue surfaces. Their Theorems 1.2 through 1.4 state the following:

\begin{theorem}[{\bf Fino-Kasuya-Vazzoni \cite{FKV}}] \label{thm2}
Let $M^n=G/\Gamma $ be a complex solvmanifold, namely, a compact quotient of a solvable Lie group by a discrete subgroup and the complex structure (when lifted onto $G$) is left-invariant. Assume that the Lie algebra ${\mathfrak g}$ of $G$ satisfies one of the following:
\begin{enumerate}
\item ${\mathfrak g}$ contains a $J$-invariant nilpotent complement.
\item $J$ is abelian. 
\item ${\mathfrak g}$ is almost abelian, either with $\dim ({\mathfrak g})=6$ or not of type (I).
\end{enumerate} 
Then Streets-Tian Conjecture holds on $M^n$. More precisely, if $M^n$ admits a Hermitian-symplectic metric, then $G$ must be abelian (hence $M$ is a complex torus).
\end{theorem}

Recall that for a given solvable Lie algebra ${\mathfrak g}$, a  nilpotent subalgebra ${\mathfrak c}\subset {\mathfrak g}$ is called a {\em nilpotent complement,} if ${\mathfrak c} + {\mathfrak n} = {\mathfrak g}$ (not necessarily a direct sum), where ${\mathfrak n}$ is the nilradical of ${\mathfrak g}$. A complex structure $J$ on a Lie algebra ${\mathfrak g}$ is said to be {\em abelian,} if $[Jx,Jy]=[x,y]$, $\,\forall\,x,y\in {\mathfrak g}$. In this case ${\mathfrak g}$ is necessarily  $2$-step solvable (\cite[Lemma 2.1]{ABD}). 

A Lie algebra ${\mathfrak g}$ is said to be {\em almost abelian} if it contains an abelian ideal of codimension $1$. The Hermitian geometry of almost abelian groups has been an active topic in recent years, see for instance \cite{AO, AL, Bock, CM, FP1, FP2, GuoZ, LW, Paradiso, Ugarte}. Here by `${\mathfrak g}$ is not of type (I)' we mean there exists $x\in {\mathfrak g}$ such that $\mbox{ad}_x$ has an eigenvalue with non-zero real part.  The extra assumption in part (iii) of Theorem \ref{thm2} was removed by Fino and Paradiso in \cite[Proposition 6.]{FP4}, so {\em  Streets-Tian Conjecture holds for all almost abelian groups.}

In \cite{FK}, utilizing the symplectic reduction theory on Lie algebras, Fino and Kasuya proved the following  beautiful theorem:

\begin{theorem}[{\bf Fino-Kasuya \cite{FK}}] \label{thm3}
Let $M^n=G/\Gamma $ be a complex solvmanifold, namely, a compact quotient of a solvable Lie group by a discrete subgroup and the complex structure (when lifted onto $G$) is left-invariant. Assume that $G$ is completely solvable. If $M^n$ admits a Hermitian-symplectic metric, then $G$ must be abelian (hence $M$ is a complex torus).
\end{theorem}

Recall that  a Lie algebra ${\mathfrak g}$ is said to be {\em completely solvable,} if $\mbox{ad}_x$ has only real eigenvalues for any $x\in {\mathfrak g}$. Clearly, completely solvable Lie algebras are always solvable, and nilpotent Lie algebras are completely solvable.

As a natural generalization to almost abelian Lie algebras, we consider Lie algebras ${\mathfrak g}$ which contain an abelian ideal ${\mathfrak a}$ of codimension $2$. The complex structures on  ${\mathfrak g}$  can be divided into two subsets: those with $J{\mathfrak a} ={\mathfrak a}$  and those with $J{\mathfrak a}\neq {\mathfrak a}$.  The second result of this article is to confirm Streets-Tian Conjecture for Lie-complex manifolds where ${\mathfrak g}$ contains an abelian ideal of codimension $2$ that is $J$-invariant:

\begin{proposition} \label{prop3}
Let $M^n=G/\Gamma $ be a compact complex manifold which is the quotient of a Lie group $G$ by a discrete subgroup $\Gamma$, where the complex structure $J$ of $M$ (when lifted onto $G$) is left-invariant. Assume that the Lie algebra ${\mathfrak g}$ of $G$ contains an abelian ideal ${\mathfrak a}$  of codimension $2$ with $J{\mathfrak a}={\mathfrak a}$. If  $M^n$ admits a Hermitian-symplectic metric, then it must admit a (left-invariant) K\"ahler metric.
\end{proposition} 

Note that when ${\mathfrak g}$ contains an abelian ideal of codimension $2$, then it is always solvable of step at most $3$, but it is not always $2$-step solvable or completely solvable. The approach is based on our previous work \cite{GuoZ}. The $J{\mathfrak a}\neq {\mathfrak a}$ case is more complicated algebraically, and in \cite{CaoZ} we were able to prove Fino-Vezzoni Conjecture (see below) for such Lie algebras, based on some partial description of balanced and pluriclosed metrics on such Lie algebras.

Lie-complex manifolds form a large set of compact complex manifolds, with adequate algebraic diversity. By focusing on the above type of special solvable Lie algebras and taking an elementary-style explicit approach, we want to better understand the algebraic complexity and subtlety in Streets-Tian Conjecture, and we hope it may shed some light towards the quest of the conjecture on other types of Lie-complex manifolds or more general non-K\"ahlerian manifolds.

Finally let us address the apparent resemblance between Streets-Tian Conjecture and the {\em Fino-Vezzoni Conjecture} (\cite{FinoVezzoni15, FinoVezzoni}), which states that a compact complex manifold admitting both a pluriclosed metric and a balanced metric must be K\"ahlerian. Sadly, even the `intersection' of the two conjectures is still open to the best of our knowledge:

\begin{conjecture} [{\bf Streets-Tian $\cap$  Fino-Vezzoni}]
If a compact complex manifold admits a Hermitian-symplectic metric and a balanced metric, then it must be K\"ahlerian.
\end{conjecture}

We mention this statement in the hope that, at least for some Lie-complex manifolds $G/\Gamma$, this weaker statement might be easier to verify than Streets-Tian or Fino-Vezzoni conjectures. 

Finally, a referee brought the following two results to our attention, which are closely related to the Streets-Tian conjecture. One is by Angella, Otal, Ugarte, Villacampa \cite[Proposition 2.2]{AOUV17}, which states that the Streets-Tian conjecture holds for any complex solvmanifold $M=G/\Gamma$ with left-invariant complex structure of splitting type on $G={\mathbb C}\ltimes {\mathbb C}^{n-1}$. The other one is due to Dinew and Popovici \cite{DP}, in which they introduced an energy functional acting on the metrics in the Aeppli cohomology class of a Hermitian-symplectic metric, and their Theorem 1.4 states that in complex dimension $n=3$ the critical points of the functional (if exists) are K\"ahler metrics.

\vspace{0.3cm}

\section{Non-balanced Bismut torsion parallel manifolds}

The goal of this section is to prove Propositions \ref{prop0} and \ref{prop1}. We will follow the notations in \cite{VYZ} and \cite{GuoZ}.  Let $(M^n,g)$ be a compact Hermitian manifold and $\omega $ be its K\"ahler form. Denote by $\nabla$ the Chern connection, and by $T$, $R$ its torsion and curvature tensor, respectively. Suppose that  $\{ e_1, \ldots , e_n\}$ is a local unitary frame and $\{ \varphi_1, \ldots , \varphi_n\}$ is the coframe of $(1,0)$-forms dual to $e$. Let $\theta$, $\Theta$ be the matrices of connection and curvature for $\nabla$ under the frame $e$, namely, $\nabla e_i = \sum_j \theta_{ij} e_j$, $\Theta_{k\ell }(e_i, \bar{e}_j)= R_{i\bar{j}k\bar{\ell}}$, then the structure equations are
$$ d\varphi = - \,^t\!\theta \wedge \varphi + \tau, \ \ \ \ \ \ \ \Theta = d\theta - \theta \wedge \theta, $$
where $\varphi$, $\tau$ are column vectors, with $\tau_i = \frac{1}{2}\sum_{j,k}T^i_{jk}\varphi_j \wedge \varphi_k$ and $T(e_j, e_k) = \sum_i T^i_{jk}e_i$. Since the K\"ahler form is $\omega = \sqrt{-1}\sum_i \varphi_i \wedge \overline{\varphi}_i$, we get from the structure equations that
\begin{equation} \label{eq:ddbaromega}
 \sqrt{-1} \partial \overline{\partial} \omega \ = \ ^t\!\tau \wedge \overline{\tau} + \,^t\!\varphi \wedge \Theta \wedge \overline{\varphi} \ = \ \frac{1}{4}\sum_{i,j,k,\ell} P^{j\ell}_{\,ik}  \, \varphi_i \wedge \varphi_k \wedge \overline{\varphi}_j \wedge \overline{\varphi}_{\ell} , 
 \end{equation}
 where
 \begin{equation}
 P^{j\ell}_{\,ik}  \ = \  \sum_r T^r_{ik} \overline{T^r_{j\ell} }  \ - R_{i\bar{j}k\bar{\ell}} +  R_{k\bar{j}i\bar{\ell}} - R_{k\bar{\ell}i\bar{j}} + R_{i\bar{\ell}k\bar{j}}.  \nonumber 
\end{equation}
Recall that a Hermitian manifold $(M^n,g)$ is said to be {\em Chern K\"ahler-like} if its Chern curvature tensor $R$ obeys the symmetry condition $R_{i\bar{j}k\bar{\ell}} = R_{k\bar{j}i\bar{\ell}}$ for any $1\leq i,j,k,\ell \leq n$. In particular, Chern flat manifolds are Chern K\"ahler-like. We are now ready to prove Proposition \ref{prop0}, which says that any compact Chern K\"ahler-like manifold, if the metric itself is not K\"ahler, cannot admit any Hermitian-symplectic or astheno-K\"ahler metric.

\begin{proof}[{\bf Proof of Proposition \ref{prop0}.}]  Let $(M^n,g)$ be a compact Hermitian manifold. Assume that $g$ is Chern K\"ahler-like and is not K\"ahler. As $g$ is Chern K\"ahler-like, its K\"ahler form $\omega$ satisfies the condition $\sqrt{-1} \partial \overline{\partial} \omega \ = \ ^t\!\tau \wedge \overline{\tau}$ by (\ref{eq:ddbaromega}). If $g_0$ is a Hermitian-symplectic metric on $M^n$, then by definition we would have a global $(2,0)$-form $\alpha$ on $M^n$ so that the $2$-form $\Omega = \alpha + \omega_0 + \overline{\alpha}$ is closed, where $\omega_0$ is the K\"ahler form of $g_0$. If we take the $(n-2,n-2)$-part of $\Omega^{n-2}$, we get
$$ (\Omega^{n-2})^{n-2,n-2} = \omega_0^{n-2} + C^2_{n-2}C^1_2 \omega_0^{n-4}\alpha \overline{\alpha} + C^4_{n-2}C^2_4 \omega_0^{n-6}(\alpha \overline{\alpha})^2 + \cdots \geq \omega_0^{n-2}. $$
Since $\Omega^{n-2}$ is closed, we have 
$$ 0=\int_M \sqrt{-1} \partial \overline{\partial} \omega \wedge \Omega^{n-2} = \int_M \,^t\!\tau \wedge \overline{\tau} \wedge (\Omega^{n-2})^{n-2,n-2} \geq  \int_M \,^t\!\tau \wedge \overline{\tau} \wedge  \omega_0^{n-2} > 0, $$
which is a contradiction, so $M^n$ cannot admit any Hermitian-symplectic metric. Similarly,  if $M^n$ admits an astheno-K\"ahler metric $g_1$ in the sense of Jost-Yau \cite{JostYau}, namely, its K\"ahler form $\omega_1$ obeys $\partial \overline{\partial} (\omega_1^{n-2})=0$, then we would have 
$$ 0=\int_M \sqrt{-1} \partial \overline{\partial} \omega \wedge \omega_1^{n-2} = \int_M \,^t\!\tau \wedge \overline{\tau} \wedge \omega_1^{n-2}  > 0, $$
again a contradiction. So $M^n$ does not admit any astheno-K\"ahler metric either. This completes the proof of the proposition.
\end{proof}

As we have mentioned in the introduction, the important special case of the above is when $(M^n,g)$ is Chern flat. In this case the result is due to  Di Scala-Lauret-Vezzoni in \cite[Proposition 3.3]{DLV}. The Chern flat case is also a consequence of the result of Fino-Kasuya-Vezzoni \cite[Corollary 3.2]{FKV} which states that compact Chern flat manifolds can not admit any pluriclosed metric unless the complex Lie group is nilpotent, and in  the nilpotent case the Streets-Tian Conjecture is already established.

\vspace{0.1cm}

Next let us recall the definition of {\em Bismut torsion parallel (BTP)} manifolds. Given a Hermitian manifold $(M^n,g)$, Bismut \cite{Bismut} proved that there exists a unique connection $\nabla^b$ on $M^n$ that is compatible with both the metric $g$ and the almost complex structure $J$, and whose torsion $T^b$ (when lowered into a tensor of type $(3,0)$ by the metric) is totally skew-symmetric. Denote by $H$ this $3$-form, namely, 
$$ H(x,y,z) = \langle T^b(x,y), z\rangle $$
for any vector fields $x,y,z$ on $M$, where $g = \langle ,\rangle $ is the metric. As is well-known, $dH=0$ if and only if $\partial \overline{\partial} \omega =0$. So Bismut connection is very natural to use when exploring Hermitian geometric questions involving pluriclosed metrics, etc.. 

As introduced in \cite{ZhaoZ22, ZhaoZ24}, we will call Hermitian manifolds satisfying $\nabla^b T^b=0$ {\em Bismut torsion parallel} manifolds, or {\em BTP manifolds} for brevity. When $n=2$, compact BTP surfaces are exactly Vaisman surfaces, which are fully classified by Belgun \cite{Belgun} in 2000. When $n\geq 3$, BTP manifolds form a relatively rich class of special Hermitian manifolds which includes all Bismut flat manifolds (and more generally, all Bismut K\"ahler-like manifolds), all Vaisman manifolds, plus other examples. A detailed description/classification of such manifolds in dimension $3$ is given in \cite{ZhaoZ22, ZhaoZ24}, but for higher dimensions this is still widely unknown. Here for our purpose of verifying the Streets-Tian Conjecture, let us first collect some basic properties for BTP manifolds that will be used later. 

Recall that {\em Gauduchon's torsion $1$-form} (\cite{Gau84}) of a given Hermitian manifold $(M^n,g)$ is the global $(1,0)$-form $\eta$ defined by $\partial (\omega^{n-1})=-\eta \wedge \omega^{n-1}$. Under a local unitary frame $e$, we have $\eta = \sum_i \eta_i \varphi_i$, $\eta_i = \sum_k T^k_{ki}$ where $\varphi$ is the local unitary coframe dual to $e$ and $T^j_{ik}$ are Chern torsion components under $e$ defined by 
$$ T(e_i, e_k) = \sum_i T^j_{ik} e_j, \ \ \ \ \ \ \forall \ 1\leq i,k\leq n. $$
As mentioned before, a Hermitian metric $g$ is said to be {\em balanced} if $d(\omega^{n-1})=0$, or equivalently, if $\eta =0$. It is a natural generalization to the K\"ahler condition which is $d\omega =0$, and balanced manifolds form an important class of special Hermitian manifolds and have been extensively studied in the past a few decades. When $n\geq 3$, BTP metrics can be balanced, although the `majority' of them tend to be non-balanced. Non-balanced BTP manifolds include Bismut K\"ahler-like and Vaisman as disjoint subsets. A distinctive property enjoyed by all non-balanced BTP manifolds is the following

\begin{lemma}[\cite{ZhaoZ24}] \label{lemma2}
Let $(M^n,g)$ be a non-balanced BTP manifold. Then given any $p\in M$, there exists a local unitary frame $e$ in a neighborhood of $p$ such that
$$ \eta = \lambda \varphi_n, \ \ \ \partial \eta =0, \ \ \ \overline{\partial} \eta =-\lambda \sum_k \overline{a}_k \varphi_k \overline{\varphi}_k,  \ \ \ T^n_{ij}=0, \ \ \ T^j_{in}= a_i\delta_{ij}, \ \ \ \ \forall \ 1\leq i,j\leq n, $$
where $\varphi$ is the coframe dual to $e$ and $\lambda$, $a_1, \ldots , a_n$ are global constants on $M^n$ satisfying $\lambda >0$, $a_n=0$, $a_1+\cdots + a_{n-1}=\lambda$. 
\end{lemma} 

A local unitary frame $e$ satisfying the conditions in Lemma \ref{lemma2} is called an {\em admissible frame} for the non-balanced BTP manifold  $(M^n,g)$. By rearranging the order of $e_i$, we may assume that $a_1\cdots a_r \neq 0$ and $a_{r+1}=\cdots =a_n=0$, where $1\leq r\leq n-1$ is an integer depends only on $g$ (the rank of the $\phi$-tensor \cite{ZhaoZ24}). We have
\begin{equation}  \label{eq:2}
(-\sqrt{-1} d\varphi_n)^r = r! \,\overline{a}_1 \cdots \overline{a}_r \,\Phi, \ \ \ \ \ \Phi = (\sqrt{-1}\varphi_1\overline{\varphi}_1)\wedge \cdots \wedge  (\sqrt{-1}\varphi_r\overline{\varphi}_r).
\end{equation}
Now we are ready to prove Proposition \ref{prop1}.

\begin{proof} [{\bf Proof of Proposition \ref{prop1}.}] Let $(M^n,g)$ be a compact, non-balanced BTP manifold. By Lemma \ref{lemma2} we know that locally there always exist admissible frames. Note that although $e$ and $\varphi$ are local (co)frames, $\varphi_n=\frac{1}{\lambda }\eta = \frac{1}{|\eta|} \eta $ is global, and  $\Phi $ is also global by formula (\ref{eq:2}). If $M^n$ admits a Hermitian-symplectic metric $g_0$, then there exists a $(2,0)$-form $\alpha$ so that $\Omega = \alpha + \omega_0 +\overline{\alpha}$ is closed, where $\omega_0$ is the K\"ahler form of $g_0$. Taking the $(n-r,n-r)$-part of $\Omega^{n-r}$, we have as in the proof of Proposition \ref{prop0} that
$$ (\Omega^{n-r})^{n-r,n-r} \geq \omega_0^{n-r}. $$
Therefore 
$$ \int_M \Phi \wedge \Omega^{n-r} = \int_M \Phi \wedge ( \Omega^{n-r})^{n-r,n-r} \geq \int_M \Phi \wedge \omega_0^{n-r} > 0. $$
On the other hand, since $\Omega$ is closed, the integral of $(-\sqrt{-1} d\varphi_n)^r \wedge \Omega^{n-r}$ is zero, so (\ref{eq:2}) would lead to a contradiction. This shows that $M^n$ cannot admit any Hermitian-symplectic metric, and Streets-Tian Conjecture holds for non-balanced BTP manifolds. This completes the proof of Proposition \ref{prop1}.
\end{proof}

\vspace{0.3cm}

\section{Hermitian-symplectic metrics on Lie algebras}

In this section, let us give the description of Hermitian-symplectic metrics on a general Lie algebra. Let us start with  {\em Lie-complex manifolds,} which means a compact complex manifold $M^n$ such that $M=G/\Gamma$ where $G$ is a Lie group, $\Gamma \subset G$ is a discrete subgroup, and the complex structure $J$ when lifted onto $G$ is left-invariant. In the past decades, the Hermitian geometry of Lie-complex manifolds were studied from various angles by A. Gray, S. Salamon, L. Ugarte, A. Fino, L. Vezzoni, F. Podest\`a, D. Angella, A. Andrada, and others. There is a vast amount of literature on this topic, here we will just mention a small sample: \cite{AU}, \cite{CFGU}, \cite{EFV}, \cite{FP3}, \cite{FinoTomassini}, \cite{FinoVezzoni}, \cite{GiustiPodesta}, \cite{Salamon}, \cite{Ugarte}, \cite{WYZ}, and interested readers could start their exploration there. For more general discussions on non-K\"ahler Hermitian geometry in general, see for example \cite{AI}, \cite{AT}, \cite{Fu}, \cite{STW}, \cite{Tosatti} and the references therein.

Here for our purpose of exploring the Streets-Tian Conjecture, first let us recall the averaging trick of Fino-Grantcharov and Ugarte \cite{FG, Ugarte}, which allows one to reduce the consideration to invariant metrics. Since the Lie group $G$ has compact quotient, it admits a bi-invariant measure \cite{Milnor}. If $g$ is a Hermitian metric with K\"ahler form $\omega$, then by averaging $\omega$ over $M$, one gets an invariant Hermitian metric $\tilde{g}$ with K\"ahler form $\tilde{\omega}$. Clearly, if $g$ is pluriclosed, then so would be $\tilde{g}$. Similarly, if $g$ is a Hermitian-symplectic metric on $M$, then there would be $(2,0)$-form $\alpha$ on $M$ so that $\Omega = \alpha + \omega + \overline{\alpha}$ is a closed. By averaging $\Omega$, we get an invariant closed $2$-form $\tilde{\Omega}$. Clearly its $(1,1)$-part is necessarily positive, thus giving an invariant Hermitian-symplectic metric. If $g$ is balanced, on the other hand, then $\omega^{n-1}$ is closed. By averaging it over, we get an invariant closed $(n-1,n-1)$-form $\Psi$, which is clearly strictly positive. Since any strictly positive $(n-1,n-1)$-form on a complex $n$-manifold would equal to $\omega_0^{n-1}$ for a unique positive $(1,1)$-form $\omega_0$, we get an invariant balanced metric as well. We refer the readers to \cite[Lemma 3.2]{EFV}) and \cite{FG, Ugarte} for more details. The point is:

\vspace{0.2cm}

 {\em In order to verify Streets-Tian Conjecture or Fino-Vezzoni Conjecture for Lie-complex manifolds, it suffices to consider invariant metrics.} 

\vspace{0.2cm}

Denote by ${\mathfrak g}$ the Lie algebra of $G$. Left invariant metrics on $G$ correspond to inner products on ${\mathfrak g}$, and left-invariant complex structures on $G$ correspond to {\em complex structure} on ${\mathfrak g}$, namely, linear (vector space) isomorphisms $J : {\mathfrak g} \rightarrow {\mathfrak g}$ satisfying $J^2=-I$ and the integrability condition 
\begin{equation} \label{integrability}
[x,y] - [Jx,Jy] + J[Jx,y] + J[x,Jy] =0, \ \ \ \ \ \ \forall \ x,y \in {\mathfrak g}. 
\end{equation}
The compatibility between a metric  $g=\langle , \rangle $ and a complex structure $J$ means as usual that $\langle Jx,Jy\rangle = \langle x,y\rangle$ for any $x,y\in {\mathfrak g }$, in other words $J$ is an orthogonal transformation.

Let ${\mathfrak g}^{\mathbb C}$ be the complexification of ${\mathfrak g}$, and write ${\mathfrak g}^{1,0}= \{ x-\sqrt{-1}Jx \mid x \in {\mathfrak g}\} \subseteq {\mathfrak g}^{\mathbb C}$. The integrability condition (\ref{integrability}) means that ${\mathfrak g}^{1,0}$ is a complex Lie subalgebra of ${\mathfrak g}^{\mathbb C}$. Extend $g=\langle , \rangle $ bi-linearly over ${\mathbb C}$, and let $e=\{ e_1, \ldots , e_n\}$ be a unitary basis of ${\mathfrak g}^{1,0}$. Following \cite{VYZ}, we will use
\begin{equation*} \label{CandD}
C^j_{ik} = \langle [e_i,e_k],\overline{e}_j \rangle, \ \ \ \ \ \  D^j_{ik} = \langle [\overline{e}_j, e_k] , e_i \rangle
\end{equation*}
to denote the structure constants, or equivalently, under the unitary frame $e$ we have
\begin{equation} \label{CandD2}
[e_i,e_j] = \sum_k C^k_{ij}e_k, \ \ \ \ \ [e_i, \overline{e}_j] = \sum_k \big( \overline{D^i_{kj}} e_k - D^j_{ki} \overline{e}_k \big) .
\end{equation}
Note that ${\mathfrak g}$ is unimodular if and only if $\mbox{tr}(ad_x)=0$ for any $x\in {\mathfrak g}$, which is equivalent to
\begin{equation} \label{unimo}
{\mathfrak g} \ \, \mbox{is unimodular}  \ \ \Longleftrightarrow  \ \ \sum_r \big( C^r_{ri} + D^r_{ri}\big) =0 , \, \ \forall \ i.
\end{equation}
Next, let us denote by $\nabla$ the Chern connection, and by $T$, $R$ its torsion and curvature tensor. Then the Chern torsion components under $e$ are given by
\begin{equation} \label{torsion}
T( e_i, \overline{e}_j)=0, \ \ \ \ T(e_i,e_j)  = \sum_k T^k_{ij}e_k, \ \ \ \  \ \ T^j_{ik}= - C^j_{ik} -  D^j_{ik} + D^j_{ki}.
\end{equation}
The structure equation takes the form:
\begin{equation} \label{eq:structure}
d\varphi_i = -\frac{1}{2} \sum_{j,k} C^i_{jk} \,\varphi_j\wedge \varphi_k - \sum_{j,k} \overline{D^j_{ik}} \,\varphi_j \wedge \overline{\varphi}_k, \ \ \ \ \ \ \forall \  1\leq i\leq n.
\end{equation}
Here $\{ \varphi_1, \ldots , \varphi_n\}$ denotes the coframe dual to $e$. Differentiate the above, we get the  first Bianchi identity, which is equivalent to the Jacobi identity in this case:
\begin{equation}
\left\{  \begin{split}  \sum_r \big( C^r_{ij}C^{\ell}_{rk} + C^r_{jk}C^{\ell}_{ri} + C^r_{ki}C^{\ell}_{rj} \big) \ = \ 0,  \hspace{3.2cm}\\
 \sum_r \big( C^r_{ik}D^{\ell}_{jr} + D^r_{ji}D^{\ell}_{rk} - D^r_{jk}D^{\ell}_{ri} \big) \ = \ 0, \hspace{3cm} \\
 \sum_r \big( C^r_{ik}\overline{D^r_{j \ell}}  - C^j_{rk}\overline{D^i_{r \ell}} + C^j_{ri}\overline{D^k_{r \ell}} -  D^{\ell}_{ri}\overline{D^k_{j r}} +  D^{\ell}_{rk}\overline{D^i_{jr}}  \big) \ = \ 0,  \end{split} \right.  \label{CC}
\end{equation}
for any $1\leq i,j,k,\ell\leq n$. To verify Streets-Tian Conjecture for  Lie-complex manifolds, we need the following characterization for Hermitian-symplectic metrics on such manifolds:
\begin{lemma} \label{lemma3}
Let $({\mathfrak g}, J, g)$ be a Lie algebra equipped with a Hermitian structure. The metric $g$ is Hermitian-symplectic if and only if for any given unitary frame $e$, there exists a skew-symmetric matrix $S$ such that
\begin{equation} \label{eq:HS}
\left\{ \begin{split} 
 \sum_r \big( S_{ri}C^r_{jk} + S_{rj} C^r_{ki}  + S_{rk} C^r_{ij} \big) =0,  \\
 \sum_r \big( S_{rk}\overline{D^i_{rj}} - S_{ri} \overline{D^k_{rj}} \big) = - \sqrt{-1}T^j_{ik}, \,     \end{split} \right. \ \ \ \ \ \ \ \ \ \ \ \ \ \forall\ 1\leq i,j,k\leq n. 
 \end{equation}
\end{lemma}

\begin{proof}
Let $\varphi$ be the unitary coframe dual to $e$. By definition, $g$ will be Hermitian-symplectic if and only if there exists a $(2,0)$-form $\alpha$ so that $\Omega = \alpha + \omega + \overline{\alpha}$ is closed. This means that $\partial \alpha =0$ and $\overline{\partial}\alpha = - \partial \omega$. Write $\alpha = \sum_{i,k} S_{ik}\varphi_i\wedge \varphi_k$, where $S$ is a skew-symmetric matrix. Then we have
\begin{eqnarray*}
 \partial \alpha & = &  2\sum_{r,k} S_{rk} \partial \varphi_r \wedge \varphi_k \ = \ - \sum_{r,i,j,k} S_{rk} C^r_{ij} \, \varphi_i\wedge \varphi_j \wedge \varphi_k \\
 & = & -\frac{1}{3} \sum_{i,j,k} \{ \sum_r \big(  S_{rk} C^r_{ij} + S_{rj} C^r_{ki} + S_{ri} C^r_{jk} \big)\}  \,\varphi_i\wedge \varphi_j \wedge \varphi_k. 
\end{eqnarray*}
So by $\partial \alpha =0$ we get the first identity in (\ref{eq:HS}). Similarly, we have
\begin{eqnarray*}
 \overline{\partial} \alpha & = &  2\sum_{r,k} S_{rk} \overline{\partial}\varphi_r \wedge \varphi_k \ = \ - 2 \sum_{r,i,j,k} S_{rk} \overline{D^i_{rj}} \, \varphi_i\wedge \overline{\varphi}_j \wedge \varphi_k \\
 & = & - \sum_{i,j,k} \{ \sum_r \big( S_{rk} \overline{D^i_{rj}} - S_{ri} \overline{D^k_{rj}} \big)\}  \, \varphi_i\wedge \overline{\varphi}_j \wedge \varphi_k. 
\end{eqnarray*}
On the other hand,
$$ \partial \omega = \sqrt{-1} \,^t\!\tau \wedge \overline{\varphi} = \sqrt{-1} \sum_{i,j,k} T^j_{ik} \,\varphi_i\wedge \varphi_k \wedge \overline{\varphi}_j , $$
so by comparing the above two lines we get the second identity in (\ref{eq:HS}). This completes the proof of Lemma \ref{lemma3}.
\end{proof}

\vspace{0.3cm}

\section{The almost abelian case revisted}

In order to make our proof of Proposition \ref{prop3} more transparent, let us start with the almost abelian case first, in which case the result is known and due to Fino-Kasuya-Vezzoni \cite[Theorem 1.4]{FKV} and Fino-Paradiso \cite[Proposition 6.1]{FP4}:

\begin{proposition} [\cite{FKV, FP4}] \label{prop2}
Let $M^n=G/\Gamma $ be a compact complex manifold which is the quotient of a Lie group $G$ by a discrete subgroup $\Gamma$, where the complex structure $J$ of $M$ (when lifted onto $G$) is left-invariant. Assume that the Lie algebra ${\mathfrak g}$ of $G$ is almost abelian, namely, it contains an abelian ideal ${\mathfrak a}$ of codimension $1$. If  $M^n$ admits a Hermitian-symplectic metric, then it must admit a (left-invariant) K\"ahler metric.
\end{proposition}

Let ${\mathfrak g}$ be a Lie algebra containing an abelian ideal ${\mathfrak a}$  of codimension $1$, and $(J,g)$ be a Hermitian structure on ${\mathfrak g}$. We will follow the notations of \cite{GuoZ}. Then ${\mathfrak a}_J:={\mathfrak a}\cap J {\mathfrak a}$ in ${\mathfrak g}$  is a $J$-invariant ideal in ${\mathfrak g}$ of codimension $2$. A unitary basis  $e$ of ${\mathfrak g}^{1,0}$ is called an {\em admissible frame} for $({\mathfrak g}, J, g)$  if ${\mathfrak a}_J$ is spanned 
by the real and imaginary parts of $e_i$ for $2\leq i\leq n$, while 
$$ {\mathfrak a} = \mbox{span}_{\mathbb R}\{ \sqrt{-1}(e_1-\overline{e}_1); \ (e_i+\overline{e}_i), \, \sqrt{-1}(e_i-\overline{e}_i); \ 2\leq i\leq n \} .$$
Let $e$ to be an admissible frame from now on. Following \cite{GuoZ}, since ${\mathfrak a}$ is abelian, we have
$$ [e_i, e_j]= [ e_i, \overline{e}_j] = [e_1-\overline{e}_1, e_i]= 0, \ \ \ \ \forall \ 2\leq i,j\leq n.$$
From this and (\ref{CandD2}) we deduce that
\begin{equation*}
C^{\ast}_{ij} = D^j_{\ast i} = D^1_{\ast i} = C^{\ast}_{1i}+\overline{D^i_{\ast 1}} =0, \ \ \ \ \forall \ 2\leq i,j\leq n.
\end{equation*}
Also, since ${\mathfrak a}$ is an ideal, $[e_1+\overline{e}_1, {\mathfrak a} ] \subseteq {\mathfrak a}$, which leads to $C^1_{1i}=0$ and $D^1_{11} = \overline{D^1_{11}}$. Putting these together, we know that the only possibly non-trivial components of $C$ and $D$ are
\begin{equation} \label{CandD-aala}
D^1_{11}=\lambda \in {\mathbb R}, \ \ \ D^1_{i1}=v_i \in {\mathbb C}, \ \ \ D^j_{i1}=A_{ij}, \ \ \ C^{j}_{1i} = - \overline{A_{ji}}, \ \ \ \ \ \  2\leq i,j\leq n.
\end{equation}
That is, for an almost abelian Hermitian Lie algebra, under any admissible frame, the only possibly non-trivial structure constants are $\lambda \in {\mathbb R}$, a complex column vector $v\in {\mathbb C}^{n-1}$, and a complex $(n-1)\times (n-1)$ matrix $A=(A_{ij})_{2\leq i,j\leq n}$. In terms of the dual coframe $\varphi$, the structure equation (\ref{eq:structure}) now takes the following form:

\begin{lemma} [\cite{GuoZ}] \label{lemma4}
Let ${\mathfrak g}$ be an almost abelian Lie algebra with a Hermitian structure $(J,g)$. Then there exists a unitary coframe $\varphi$ under which the structure equation is
\begin{equation*}
\left\{ \begin{split}  d\varphi_1 = - \lambda \,\varphi_1\wedge \overline{\varphi}_1 , \hspace{5.8cm} \\ d\varphi_i = - \overline{v}_i \, \varphi_1\wedge \overline{\varphi}_1 +  \sum_{j=2}^n \overline{A_{ij}} \,(\varphi_1 + \overline{\varphi}_1)\wedge \varphi_j , \ \ \ 2\leq i\leq n. \end{split} \right.
\end{equation*}
\end{lemma}

Also, from part (ii) and (iii) of \cite[Proposition 7]{GuoZ}, we know that 
\begin{lemma} [\cite{GuoZ}] \label{lemma5}
Let ${\mathfrak g}$ be an almost abelian Lie algebra with a Hermitian structure $(J,g)$. Then under any admissible frame $e$, it holds that
\begin{enumerate}
\item   ${\mathfrak g}$ is unimodular $\ \Longleftrightarrow \ $ $\lambda + \mbox{tr}(A) + \overline{\mbox{tr}(A)} = 0$;
\item $g$ is K\"ahler  $\ \Longleftrightarrow \ $ $v=0$ and $A+A^{\ast}=0$.
 \end{enumerate}
\end{lemma}
Here $A^{\ast}$ stands for the conjugate transpose of $A$.  Note that if a Lie group admits a compact quotient, then its Lie algebra is necessarily unimodular. We need one more piece of information before we will proceed with the proof of Proposition \ref{prop2}. Suppose $g$, $\tilde{g}$ are two Hermitian metrics on almost abelian $({\mathfrak g},J)$. Let $e$ be an admissible frame for $g$. By a unitary change of $\{ e_2, \ldots , e_n\}$ if necessary, we may assume that
\begin{equation} \label{echange-aala}
\tilde{e}_1=pe_1 + \sum_{k=2}^n a_k e_k, \ \ \ \tilde{e}_i=p_ie_i,  \ \ \  2\leq i\leq n,   
\end{equation}
is an admissible frame for $\tilde{g}$, where $p$, $p_i$, $a_i$ are constants with each $p_i>0$. By rotating $\tilde{e}_1$ if needed, we may also assume that $p>0$. Their dual coframes are related by
$$ \tilde{\varphi}_1 = \frac{1}{p} \varphi_1, \ \ \ \ \tilde{\varphi}_i = \frac{1}{p_i} \varphi_i - \frac{a_i}{pp_i}\varphi_1 . $$
Plug them into the structure equation in Lemma \ref{lemma4}, we get the formula relating the structure constants of the two metrics:
\begin{equation}  \label{Achange-aala}
\tilde{\lambda }= p\lambda , \ \ \ \tilde{A}=pP^{-1}AP, \ \ \ \tilde{v} = p^2 P^{-1} v - p\lambda P^{-1}\overline{a} + p P^{-1}A \overline{a}, 
\end{equation}
where $P = \mbox{diag}\{ p_2, \ldots , p_n\}$. In summary, we have

\begin{lemma}  \label{lemma5b}
Let ${\mathfrak g}$ be an almost abelian Lie algebra with a complex structure $J$. Given any two Hermitian metrics  $g$, $\tilde{g}$ on $({\mathfrak g},J)$,  there exist admissible frames $e$ for $g$ and $\tilde{e}$ for $\tilde{g}$ related by (\ref{echange-aala}), with the corresponding structure constants related by (\ref{Achange-aala}).  
\end{lemma}

Now we are ready to prove Proposition \ref{prop2}, which by the averaging trick of Fino-Grantcharov is equivalent to the following
\begin{proposition}  \label{prop4}
Let ${\mathfrak g}$ be a unimodular, almost abelian Lie algebra and $J$ a complex structure on ${\mathfrak g}$. If $({\mathfrak g},J)$ admits a Hermitian-symplectic metric, then it will admit a K\"ahler metric. 
\end{proposition}

\begin{proof}
Let $({\mathfrak g},J,g)$ be a unimodular, almost abelian Lie algebra with  Hermitian structure. Assume that $g$ is Hermitian-symplectic. Let $e$ be an admissible frame. Then by Lemma \ref{lemma3}, the structure constants $C$ and $D$ obey the two identities in Lemma \ref{lemma3} for some skew-symmetric $n\times n$ matrix $S$. Write
$$ S=\left[ \begin{array}{cc} 0 & -\,^t\!u \\ u & \, S' \end{array} \right] , $$
where $u\in {\mathbb C}^{n-1}$ is a column vector and $S'$ is a skew-symmetric $(n-1)\times (n-1)$ matrix. Plug the data (\ref{CandD-aala}) into the two identities of Lemma \ref{lemma3}, we obtain
\begin{eqnarray}
&& S'\overline{A}+A^{\ast}S' =0, \label{eq:13} \\
&& A + A^{\ast} = 0,   \label{eq:14} \\
&& S'\overline{v} + A^{\ast} u + \lambda u = \sqrt{-1} v. \label{eq:15}
\end{eqnarray}
Since ${\mathfrak g}$ is assumed to be unimodular, by part (i) of Lemma \ref{lemma5}  and (\ref{eq:14}) we know that $\lambda =0$. So equation (\ref{eq:15}) becomes
\begin{equation} \label{eq:a}
S'\overline{v}-\sqrt{-1}v = Au.
\end{equation}
Taking the conjugate, and then mutiplying by $\sqrt{-1}S'$ from the left, we get
$$ \sqrt{-1}S' \overline{S'} v - S' \overline{v} = \sqrt{-1}S'\overline{A} \overline{u} = \sqrt{-1} AS'\overline{u} . $$
Note that in the last equality we used (\ref{eq:13}) and (\ref{eq:14}). Add the above equality to (\ref{eq:a}), we get
 \begin{equation} \label{eq:b}
- \sqrt{-1} (I - S'\overline{S'})v = A( u + \sqrt{-1}S'\overline{u}) .
\end{equation}
From (\ref{eq:13}) and (\ref{eq:14}), we have $S'\overline{A}=AS'$. Taking conjugate, we get  $\overline{S'}A=\overline{A}\,\overline{S'}$. Hence
$$ S'\overline{S'}A = S' (\overline{A}\,\overline{S'}) = (S' \overline{A})\overline{S'} = AS' \overline{S'}. $$
So $I-S'\overline{S'}$ commutes with $A$. Since $S'$ is skew-symmetric, $I-S'\overline{S'}>0$. Thus by (\ref{eq:b}) we get 
$$ v=  (I-S'\overline{S'})^{-1}A(\sqrt{-1}u - S'\overline{u})= - A (I-S'\overline{S'})^{-1}(  S'\overline{u}- \sqrt{-1}u). $$
Write $\overline{a}=(I-S'\overline{S'})^{-1}(S'\overline{u} - \sqrt{-1}u )$, the above equation tells us that $v=-A\overline{a}$ belongs to the image space of $A$. Now let $\tilde{e}_1=e_1+\sum_{k=2}^n a_ke_k$ and  $\tilde{e}_i=e_i$ for $2\leq i\leq n$. By using $\tilde{e}$ to be the unitary frame, we get a new Hermitian metric $\tilde{g}$ on $({\mathfrak g},J)$, whose structure constants by Lemma \ref{lemma5b} are given by 
$$ \tilde{\lambda} =\lambda =0, \ \ \ \tilde{A}=A \ \, \mbox{is skew-Hermitian, \ and} \ \ \tilde{v} = v + A \overline{a} = 0. $$ 
By part (ii) of Lemma \ref{lemma5} we now know that the new Hermitian metric $\tilde{g}$ is K\"ahler. This completes the proof of Proposition \ref{prop4}.
\end{proof}

\vspace{0.3cm}

\section{Lie algebras with $J$-invariant abelian ideal of codimension $2$}

In this section, we will prove Proposition \ref{prop3}, which is the main technical result of this article. Let ${\mathfrak g}$ be a Lie algebra and ${\mathfrak a}\subseteq {\mathfrak g}$ be an abelian ideal of codimension $2$. Note that such ${\mathfrak g}$ will always be solvable of step at most $3$, but in general it will not be $2$-step solvable. We should mention that $2$-step solvable groups form an interesting and rich class of study. In the beautiful recent papers \cite{FSwann},\cite{FSwann2}, Freibert and Swann studied the Hermitian geometry of such Lie groups, especially on balanced and pluriclosed (namely, SKT) structures.

Let $(J,g)$ be a Hermitian structure on ${\mathfrak g}$. Then the codimension of ${\mathfrak a}_J:= {\mathfrak a}\cap J{\mathfrak a}$ is either $2$ or $4$, and it is $2$ if and only if $J{\mathfrak a}={\mathfrak a}$. As mentioned before, the Hermitian structures in the $J{\mathfrak a}\neq {\mathfrak a}$ case are  much more complicated from the algebraic point of view, and in this article we will restrict ourselves to the easier case. We will assume $J{\mathfrak a}={\mathfrak a}$  from now on. 

First, let us follow the notations of \cite{GuoZ} to select simple frames and write down the structure constants. A unitary basis $\{ e_1, \ldots , e_n\}$ of ${\mathfrak g}^{1,0}$ will be called an {\em admissible frame} for $({\mathfrak g},J,g)$ if 
$$ {\mathfrak a} = \mbox{span}_{\mathbb R} \{ e_i+\overline{e}_i, \, \sqrt{-1}(e_i-\overline{e}_i); \ 2\leq i\leq n\} . $$
Since ${\mathfrak a}$ is abelian and is an ideal, we have
\begin{equation*}
C^{\ast}_{ij}=D^j_{\ast i} =C^1_{\ast \ast} = D^{i}_{1\ast} =D^{\ast}_{1i} = 0, \ \ \ \ \ \forall \ 2\leq i,j\leq n,
\end{equation*}
and the only possibly non-zero components of $C$ and $D$ are
\begin{equation}  \label{CD-XYZ}
C^{j}_{1i}= X_{ij}, \ \ \ D^1_{11}=\lambda, \ \ \ D^j_{i1}=Y_{ij}, \ \ \ D^1_{ij}=Z_{ij}, \ \ \ D^1_{i1}=v_i, \ \ \ \ \ 2\leq i,j\leq n,
\end{equation}
where $\lambda \geq 0$, $v\in {\mathbb C}^{n-1}$ is a column vector and $X$, $Y$, $Z$ are $(n-1)\times (n-1)$ complex matrices. Here and from now on we have rotated the angle of $e_1$ to assume that $D^1_{11}\geq 0$.  Let $\varphi$ be the coframe dual to $e$. For convenience, let us write $\varphi'$ for the column vector $^t\!(\varphi_2, \ldots , \varphi_n)$, then the structure equation becomes 
\begin{equation}
\left\{ \begin{split} d\varphi_1 \, = \, -\lambda \varphi_1\overline{\varphi}_1 , \hspace{4.4cm} \\
d\varphi' \, = \, - \varphi_1\overline{\varphi}_1 \overline{v} -\varphi_1\,^t\!X \varphi' + \overline{\varphi}_1 \overline{Y} \varphi' - \varphi_1 \overline{Z} \,\overline{\varphi'}. \end{split} \right.  \label{structure}
\end{equation}
By the Bianchi identity (\ref{CC}), or equivalently, by $d^2\varphi =0$, we get a system of matrix equations
\begin{equation}  
\left\{ \begin{split} \lambda (X^{\ast}\!+Y)+ [X^{\ast} ,Y]  -  Z\overline{Z} \, = \, 0, \\
\lambda Z - ( Z \,^t\!X + Y Z )\, =\, 0. \hspace{1.2cm} \end{split} \right.   \label{XYZ}
\end{equation}
Also, by (\ref{torsion}), we know that the possibly non-zero components of the Chern torsion are:
\begin{equation} \label{torsion-XYZ}
T^1_{1i}=v_i, \ \ \ \ T^1_{ij}=Z_{ji}-Z_{ij}, \ \ \ \ T^j_{1i}=Y_{ij}-X_{ij}, \ \ \ \ \ 2\leq i,j\leq n.
\end{equation}
By part (i) and (iii) of \cite[Proposition 3]{GuoZ}, we have
\begin{lemma} [\cite{GuoZ}]  \label{lemma7}
Let $({\mathfrak g},J,g)$ be a Lie algebra with Hermitian structure and let ${\mathfrak a}\subseteq {\mathfrak g}$ be a $J$-invariant abelian ideal of codimension $2$. Then under any admissible frame,
\begin{enumerate}
\item ${\mathfrak g}$ is unimodular if and only if $\, \lambda - \mbox{tr}(X)+\mbox{tr}(Y)=0$,
\item $g$ is K\"ahler if and only if $v=0$, $\,^t\!Z=Z$, and $X=Y$. When ${\mathfrak g}$ is unimodular, $g$ is K\"ahler if and only if $\lambda =0$, $v=0$, $Z=0$, and $X=Y$ is normal. 
\end{enumerate}
\end{lemma}

Analogous to the almost abelian case in the previous section, here we will need one more lemma before we can proceed with the proof of Proposition \ref{prop3}. Suppose $g$ and $\tilde{g}$ are both Hermitian metrics on $({\mathfrak g}, J)$. Let $e$ be a given admissible frame for $g$. Rotate $\{ e_2, \ldots , e_n\}$ by a unitary matrix if necessary, we may assume that $\tilde{e}$ given below is an admissible frame for $\tilde{g}$:
\begin{equation}  \label{echange}
\tilde{e}_1 =p_1 e_1 + \sum_{k=2}^n a_ke_k, \ \ \ \ \ \ \tilde{e}_i = p_ie_i ,\ \  (2\leq i\leq n).
\end{equation}
Here $p_1$, $p_i$ are positive constants, and $a_i$ are complex numbers. Scale $\tilde{g}$ by a positive constant if necessary, we may always assume that $p_1=1$, which will simplify our writing below. The dual coframes are related by
$$ \tilde{\varphi}_1=\varphi_1, \ \ \ \ \tilde{\varphi}_i=\frac{1}{p_i}(\varphi_i-a_i\varphi_1), \ \ (2\leq i\leq n). $$
From the structure equation (\ref{structure}), by a straight forward computation we get the following relations between the two sets of structure constants:
\begin{equation} \label{XYZchange}
\left\{  \begin{split} \tilde{\lambda}=\lambda, \ \  \tilde{X}=PXP^{-1}, \  \ \tilde{Y}=P^{-1}YP, \ \ \tilde{Z}=P^{-1}ZP, 
\\ \tilde{v} = v -\lambda \overline{a} + P^{-1}Y \overline{a} +  P^{-1} Z a, \hspace{3.2cm}
\end{split} \right. 
\end{equation}
where $P=\mbox{diag}\{ p_2, \ldots , p_n\}$. In summary, we have
\begin{lemma}   \label{lemma8}
Let ${\mathfrak g}$ be a Lie algebra with a left-invariant complex structure $J$, and ${\mathfrak a}\subseteq {\mathfrak g}$ be a $J$-invariant abelian ideal of codimension $2$. Suppose $g$ and $\tilde{g}$ are both metrics on ${\mathfrak g}$ compatible with $J$. Then there exist admissible frames $e$ for $g$ and  $\tilde{e}$ for $\tilde{g}$ so that they are related by (\ref{echange}). Assuming that $p_1=1$, which can always be achieved by scaling $\tilde{g}$, then the corresponding structure constants are related by formula (\ref{XYZchange}). 
\end{lemma}

With all these preparations at hand, now we are ready to prove Proposition \ref{prop3} stated in the introduction. 

\begin{proof} [{\bf Proof of Proposition \ref{prop3}:}]
Let ${\mathfrak g}$ be a unimodular Lie algebra with an abelian ideal ${\mathfrak a}$ of codimension $2$ and with a Hermitian structure $(J,g)$. We assume that $J{\mathfrak a}={\mathfrak a}$ and let $e$ be an admissible frame. Assume that the metric $g$ is Hermitian-symplectic. Then the two identities in Lemma \ref{lemma3} hold for some skew-symmetric matrix $S$. Again write
$$ S=\left[ \begin{array}{cc} 0 & -\,^t\!u \\ u & \, S' \end{array} \right] , $$
where $S'$ is a skew-symmetric $(n-1)\times (n-1)$ matrix. By plugging in the data (\ref{CD-XYZ}), we get the following
\begin{eqnarray}
&& XS'+S'\,^t\!X=0, \label{eq:24} \\
&& S' \overline{Z} = \sqrt{-1} (Y-X),  \label{eq:25} \\
&& S' \overline{Y} + Y^{\ast}S' = \sqrt{-1} (Z-\,^t\!Z), \label{eq:26} \\
&& S'\overline{v} +  Y^{\ast}u +\lambda u =  \sqrt{-1}v.  \label{eq:27}
\end{eqnarray}
The matrices $X$, $Y$, $Z$ also satisfy (\ref{XYZ}) and $\lambda = \mbox{tr}(X-Y)$ by part (i) of Lemma \ref{lemma7}. Starting with the equation (\ref{eq:26}), multiply by $\overline{Z}$ on the right, we have
\begin{eqnarray}
Z\overline{Z}-\,^t\!Z\overline{Z} & = & -\sqrt{-1} (S'\overline{Y}\overline{Z} + Y^{\ast}S'\overline{Z}) \nonumber \\
& = &  -\sqrt{-1}\{ S' \overline{(\lambda Z - Z\,^t\!X)} +  Y^{\ast}S'\overline{Z}\} \nonumber \\
& = & -\sqrt{-1} \{  \lambda  (S' \overline{Z})- ( S' \overline{Z})X^{\ast} + Y^{\ast} (S' \overline{Z})\} \nonumber \\
& = & \lambda (Y-X) - (Y-X)X^{\ast} + Y^{\ast} (Y-X) . \label{eq:ZZt}
\end{eqnarray}
Here at the second equality sign we used the second equation of (\ref{XYZ}) and at the fourth equality we used  (\ref{eq:25}). Compare (\ref{eq:ZZt}) with the first line of (\ref{XYZ}), we obtain the following
\begin{equation*} 
^t\!Z\overline{Z} = \lambda (X+X^{\ast}) - XX^{\ast} - Y^{\ast}Y + X^{\ast}Y  + Y^{\ast} X, 
\end{equation*}
or equivalently, 
\begin{equation} \label{eq:29}
^t\!Z\overline{Z}  +(X-Y)^{\ast} (X-Y) = \lambda (X+X^{\ast}) + [X^{\ast}, X]  .
\end{equation}
The next step is to show that $\lambda =0$. By (\ref{eq:26}), we know that $Z+\sqrt{-1}S'\overline{Y}$ is a symmetric matrix. Let us denote this matrix by $A$, so $Z = A - \sqrt{-1}S'\overline{Y}$. By (\ref{eq:25}), we have
$$ X - Y = \sqrt{-1}S'\overline{Z} =  \sqrt{-1}S'(\overline{A} + \sqrt{-1} \overline{S'} Y) = S' S'^{\ast} Y + \sqrt{-1}S'  \overline{A} $$
Since $S'$ is skew-symmetric and $A$ is symmetric, the trace of $S' \overline{A}$ is zero, so we get
\begin{equation} \label{eq:tr}
\lambda = \mbox{tr}(X-Y) = \mbox{tr}(S' S'^{\ast} Y ). 
\end{equation}
Write $H=I + S' S'^{\ast} > 0$. Plug $X=HY+ \sqrt{-1}S'  \overline{A}$ into (\ref{eq:24}), we get
$$ HYS' + S'\,^t\!Y\,^t\!H = 0. $$
That is, $HYS'=B$ is symmetric. Since $^t\!H=\overline{H} = I + S'^{\ast} S'$, we have $S'\,^t\!H=HS'$, hence $S'^{\ast}H^{-1}$ is skew-symmetric, thus by (\ref{eq:tr}) we get
$$ \lambda =  \mbox{tr}(S' S'^{\ast} Y ) = \mbox{tr}(S'^{\ast} Y S') = \mbox{tr}(S'^{\ast}H^{-1}B) =0,   $$
since $B$ is symmetric and $S'^{\ast}H^{-1}$ is skew-symmetric. Now by taking trace on both sides of (\ref{eq:29}), we get $Z=0$ and $X=Y$, and $[X^{\ast},X]=0$ as well. In summary, we have:

\vspace{0.1cm}

{\em For a unimodular Hermitian Lie algebra $({\mathfrak g}, J, g)$ which contains a $J$-invariant abelian ideal of codimension $2$, the metric $g$ will be Hermitian-symplectic if and only if $\lambda =0$, $Z=0$, $X=Y$ is normal, and $v$ satisfies
\begin{equation} \label{eq:v}
S'\overline{v}+X^{\ast} u = \sqrt{-1} v
\end{equation}
for some column vector $u$ and skew-symmetric matrix $S'$. }

\vspace{0.1cm}

We claim that condition (\ref{eq:v}) is equivalent to $v \in \mbox{Im}(X)$, namely, $v$ belongs to the image space of $X$. To see this, since $X$ is normal, we may rotate our unitary basis $\{ e_2, \ldots , e_n\}$ to assume that $X$ is diagonal. Permute these $e_i$ if needed, we may assume that 
$$ X =\left[ \begin{array}{cc} X_1 & 0 \\ 0 & 0 \end{array} \right] ,  $$
where $X_1$ is non-degenerate (and  diagonal). By (\ref{eq:24}), $S'$ must be in corresponding block-diagonal form
$$ S' =\left[ \begin{array}{cc} S_1 & 0 \\ 0 & S_2 \end{array} \right] ,  $$
with $S_1$, $S_2$ being skew-symmetric. Now equation (\ref{eq:v}) becomes
$$ \left[ \begin{array}{c} S_1\overline{v}_1  \\  S_2\overline{v}_2 \end{array} \right] + \left[ \begin{array}{c} X_1^{\ast}u_1  \\  0 \end{array} \right] 
= \sqrt{-1} \left[ \begin{array}{c} v_1  \\  v_2 \end{array} \right] 
.  $$
Multiplying $v_2^{\ast}$ from the left side on the lower block of the above line, we get $0=\sqrt{-1}|v_2|^2$ as $S_2$ is skew-symmetric. This means that $v_2=0$ so $v$ is in the image space of $X$, namely, there exists a column vector $a$ so that $v=-X\overline{a}$. Now for
$$ \tilde{e}_1 = e_1 + \sum_{k=2}^n a_k e_k, \ \ \ \tilde{e}_i =e_i,\ \ 2\leq i\leq n, $$
then by Lemma \ref{lemma8} we know that the metric $\tilde{g}$ on $({\mathfrak g},J)$ with $\tilde{e}$ being unitary would have structure constants
$$ \tilde{\lambda}=0, \ \ \ \tilde{Z}=0, \ \ \ \tilde{X}=\tilde{Y} \ \ \mbox{is normal, \ \ and} \ \tilde{v}=0. $$ 
Hence the new Hermitian metric $\tilde{g}$ is  K\"ahler by part (ii) of Lemma \ref{lemma7}, and we have completed the proof of Proposition \ref{prop3}.
\end{proof}

\vspace{0.3cm}

\vs

\noindent\textbf{Acknowledgments.} We would like to thank Hisashi Kasuya for bringing the references \cite{FK} and \cite{FKV} to our attention, which covered a number of interesting types of solvmanifolds. The second named author would like to thank Bo Yang and Quanting Zhao for their interest and helpful discussions. We would also like to thank the referees are a number of valuable suggestions and corrections which enhanced the readability of the manuscript. 

\vs

\end{document}